\documentclass[11pt,a4paper]{amsart}
\usepackage{graphicx,cancel}
\usepackage{subfigure}
\usepackage{amsmath,amssymb,amsthm,bbm,marvosym}
\usepackage{enumitem}
\setenumerate{leftmargin=*}

\usepackage{xcolor}
\usepackage[margin=2cm,bottom=20mm,footskip=10mm,top=20mm]{geometry}

\let\subset\subseteq 
\let\eps\varepsilon
\let\rho\varrho

\def\NN{\mathbb{N}}

\newcommand{\JUSTIFY}[1]{\fbox{\tiny{#1}}\quad}

\def\tower{\mathsf{tow}}
\def\NIn{\mathsf{N}^\mathsf{in}}
\def\NOut{\mathsf{N}^\mathsf{out}}

\newtheorem{theorem}{Theorem}
\newtheorem{lemma}[theorem] {Lemma}

\newtheorem{fact}[theorem]{Fact} 
\newtheorem{observation}[theorem] {Observation}
\newtheorem{construction}[theorem] {Construction} 

\newtheorem{claimNewNumbering}{Claim}
\newtheorem{claim}[claimNewNumbering]{Claim}
\newtheorem*{claim*}{Claim}
\newtheorem*{fact*}{Fact} 

\newtheorem{definition}[theorem] {Definition} 
\newtheorem*{remark*} {Remark}

\newcommand{\By}[2]{\overset{\mbox{\tiny{#1}}}{#2}}
\newcommand{\ByRef}[2]{   \By{\eqref{#1}}{#2} }

\newcommand{\eqByRef}[1]{ \ByRef{#1}{=} }

\newcommand{\geByRef}[1]{ \ByRef{#1}{\ge} }

\renewcommand{\leq}{\leqslant}
\renewcommand{\le}{\leqslant}
\renewcommand{\geq}{\geqslant}
\renewcommand{\ge}{\geqslant}

\renewcommand{\epsilon}{\varepsilon}

\title{A tower lower bound for the degree relaxation of the Regularity Lemma}
\author{Frederik Garbe}
\address{Heidelberg University, Faculty of Mathematics and Computer Science, Im Neuenheimer Feld 205, 
69120 Heidelberg, Germany}
\email{garbe@informatik.uni-heidelberg.de}
\author{Jan Hladký}
\address{Institute of Computer Science of the Czech Academy of Sciences, Pod Vod\'{a}renskou v\v{e}\v{z}\'{\i} 2, 182~07 Prague, Czechia. With institutional support RVO:67985807.}
\email{hladky@cs.cas.cz}
\thanks{Research supported by Czech Science Foundation Project 21-21762X (J. Hladk\'y) and by the Deutsche Forschungsgemeinschaft (DFG, German Research Foundation) - 428212407 (F. Garbe).}
\date{}

\begin{document}
\begin{abstract}
It is well-known that if $(A,B)$ is an $\tfrac{\eps}{2}$-regular pair (in the sense of Szemerédi) then there exist sets $A'\subset A$ and $B'\subset B$ with $|A'|\le \eps|A|$ and $|B'|\le \eps|B|$ so that the degrees of all vertices in $A\setminus A'$ differ by at most $\eps|B|$ and the degrees of all vertices in $B\setminus B'$ differ by at most $\eps|A|$. We call such a property \emph{$\eps$-degularity}. This leads to the notion  of an \emph{$\eps$-degular partition} of a graph in the same way as the definition of $\eps$-regular pairs leads to the notion of $\eps$-regular partitions. 

We show that there exist graphs in which any $\eps$-degular partition requires the number of clusters to be $\mathrm{tower}(\Theta(\eps^{-1/3}))$. That is, even though degularity is a substantial relaxation of regularity, in general one cannot improve much on the bounds that come with Szemerédi's regularity lemma.
\end{abstract}
\maketitle
\section{Introduction}
\subsection{Regularity Lemma}
A key concept of the famous Regularity Lemma of Szemerédi~\cite{Sze:ReguLemma} from 1976 is that of a regular pair. To this end, we recall that for two nonemtpy sets $X$ and $Y$ of vertices of a graph $G$, the \emph{density of $(X,Y)$} is defined as $d_G(X,Y)=\frac{e(X,Y)}{|X||Y|}$. A pair $(A,B)$ of disjoint nonempty sets of vertices is \emph{$\epsilon$-regular} if 
\begin{equation}\label{eq:defregular}
\left|d_G(A,B)-d_{G}(A',B')\right|\le \eps\;\mbox{for every $A'\subset A$ and $B'\subset B$ with $|A'|>\eps|A|$ and $|B'|>\eps|B|$.}
\end{equation}
Regular pairs are the building blocks of regular partitions. More precisely, if $G$ is a graph, we say that a partition $V(G)=V_1\sqcup \ldots \sqcup V_\ell$, where the sets $V_1,\ldots,V_\ell$ are called \emph{clusters}, is an \emph{$\eps$-regular partition of $G$ of complexity $\ell$} if we have
\begin{enumerate}[label=(P\arabic*)]
	\item\label{P:2} $\big| |V_i|-|V_j|\big|\le 1$ for every $i,j\in[\ell]$.
	\item\label{P:3} For each $i\in [\ell]$ for at least $(1-\eps)\ell$ indices $j\in[\ell]\setminus\{i\}$ we have that $(V_i,V_j)$ is an $\eps$-regular pair.
\end{enumerate}
We note that several slightly different, but essentially equivalent, definitions of regular partitions appear in the literature, each of which may have advantages in specific contexts. We discuss two such variants in Sections~\ref{ssec:May1} and~\ref{ssec:May2}.

Observe that~\ref{P:3} forces that 
\begin{equation}\label{eq:Eindhoven}
    \ell\ge \eps^{-1}\;,
\end{equation}
which is often a convenient lower bound for applications of the regularity method. 

The Regularity Lemma then says that for each $\eps>0$ the number
\begin{equation*}
R(\eps):=\sup\{\min\{\ell\::\:\text{$G$ has an $\eps$-regular partition of complexity $\ell$}\}\;:\;\text{$G$ is a graph}\}
\end{equation*}
is finite. This statement, which has been central to extremal graph theory, information theory, and theoretical computer science, naturally, calls for studying the value $R(\eps)$. Szemerédi's proof gives an upper bound $R(\eps)\le \tower(\Theta(\eps^{-5}))$. Here,  $\tower:[0,\infty)\to\NN$ is the \emph{tower function}, defined as $\tower(x)=1$ for $x\in[0,1)$ and $\tower(1+x)=2^{\tower(x)}$ for $x\ge 1$.\footnote{This operation is sometimes also called `tetration'.}

It was a major question to obtain any nontrivial lower bound on $R(\eps)$. A breakthrough was achieved by Gowers~\cite{MR1445389} who obtained a qualitatively matching lower bound $R(\eps)>\tower(\Theta(\eps^{-1/16}))$. In fact, Gowers's paper has two parts. In the first one, he gives a bound $R(\eps)>\tower(\Theta(\log(\frac1\eps)))$ using a fairly short argument. In the second part, which was called a \emph{tour de force} by Bollobás in his official description of the work for Gowers's Fields medal~\cite{MR1660652}, he then proves the already mentioned bound $R(\eps)>\tower(\Theta(\eps^{-1/16}))$. Subsequently, Gowers's construction was streamlined by Moshkovitz and Shapira~\cite{MR3516883} who based on the previously mentioned short argument proved the lower bound $R(\eps)>\tower(\Theta(\eps^{-1/2}))$.
The state of the art is~\cite{MR3737374} in which a tight bound $R'(\epsilon)=\tower(\Theta(\eps^{-2}))$, where $R'(\epsilon)$ is a slight variant of $R(\epsilon)$ defined using the `aggregate form' described in Section~\ref{ssec:May1}. 
Note that \cite{MR1445389} and \cite{MR1723039} obtain lower-bounds in a finer setting in which the three different roles of $\epsilon$ --- the tolerance of density deviation in~\eqref{eq:defregular}, the lower bound on the size of the sets $A'$ and $B'$ in~\eqref{eq:defregular}, and the proportion of regular pairs quantified in~\ref{P:3} --- are considered as separate parameters.

The importance of the Regularity Lemma (and its most important consequence, the Removal Lemma) initiated a whole line of research and led to many decomposition statements which are more or less related. One of the most prominent ones is the Weak Regularity Lemma of Frieze and Kannan~\cite{MR1723039}. The main upshot of the Weak Regularity Lemma is that the number of clusters in this type of regularization is guaranteed to be only $2^{O(\eps^{-2})}$, and this bound is known to be almost tight,~\cite{MR3811511}. Quantitative aspects of other regularity lemmas are obtained and surveyed in~\cite{MR2989432}.
Another recent example how to achieve a partition with a substantially smaller number of cluster, but for a weaker regularity-type property, can be found in~\cite{Cs21}. 

\subsection{Regularity versus degularity}
The usage of the word \emph{regular} is somewhat unfortunate as it traditionally refers to Szemer\'edi's concept of partition's and at the same time to another property of all the degrees being equal.\footnote{Some authors in the 1990's used the term \emph{Uniformity lemma} instead of the Regularity Lemma, see e.g.~\cite{MR1204114,MR1445389}.} Throughout the paper, the term \emph{regular pair} or \emph{regular partition} refers to to the former concept whereas the term \emph{regular graph} refers to the later. In this paper we study the link between these two properties.

We call a pair $(A,B)$ of disjoint vertex sets of a graph \emph{$\epsilon$-degular} if there exist sets $A^-\subset A$ and $B^-\subset B$ with $|A^-|\le \eps|A|$ and $|B^-|\le \eps|B|$ such that
\begin{align}
	\begin{split}
		\label{eq:defdegular}
		\left|\deg(a,B)-\frac{e(A,B)}{|A|}\right|\le \eps|B| 
  \mbox{ and }
  \left|\deg(b,A)-\frac{e(A,B)}{|B|}\right|\le \eps|A|\;,
	\end{split}
\end{align}
for every $a\in A\setminus A^-$ and $b\in B\setminus B^-$.

It is well-known that an $\eps$-regular pair is $2\eps$-degular. For completeness we recall a proof of this fact.
\begin{fact}\label{fact:regulardegular}
	Every $\eps$-regular pair $(A,B)$ is also $2\eps$-degular. In particular, for the sets $A_1:=\{a\in A: \deg(a,B)<\frac{e(A,B)}{|A|}-\eps|B|\}$ and $A_2:=\{a\in A: \deg(a,B)>\frac{e(A,B)}{|A|}+\eps|B|\}$ we have $|A_1|,|A_2|\le \eps|A|$. By symmetry, we have a similar property for $B$.
\end{fact}
\begin{proof}
	We will only prove that $|A_1|\le \eps|A|$. The same argument can be used to bound $|A_2|$. 
	If $A_1\neq\emptyset$ then
	\begin{align*}
		\frac{e(A,B)}{|A||B|}-\frac{e(A_1,B)}{|A_1||B|}
		=
		\frac{e(A,B)}{|A||B|}-\frac{\sum_{a\in A_1}\deg(a,B)}{|A_1||B|}
		>
		\eps\;.
	\end{align*}
	If we had $|A_1|> \eps|A|$ then we could use the pair  $(A_1,B)$ in~\eqref{eq:defregular} as a witness of $\eps$-irregularity of $(A,B)$, which would be a contradiction.
\end{proof}

If we replace in~\ref{P:3} the notion of $\eps$-regular pairs by that of $\eps$-degular pairs then~\ref{P:2} and~\ref{P:3} lead to the notion of an \emph{$\eps$-degular partition of $G$ of complexity $\ell$}.

As we saw, regularity is at least as strong as degularity. But in fact, often it is much stronger. To see this, consider a pair $(A,B)$ where $A=A^1\sqcup A^2$ and $B=B^1\sqcup B^2$ with $|A^1|=|A^2|$ and $|B^1|=|B^2|$ such that for $i=1,2$, the pair $(A^i,B^i)$ is complete and the pair $(A^i,B^{3-i})$ is edgeless. The pair $(A,B)$ is $0$-degular while it is not even $0.499$-regular. Indeed, this example can be used to show that key features from the theory of the Regularity Lemma, such as the Blow-up lemma or the Counting lemma, fail badly when relaxed to degular pairs. This example also shows that the key property that a non-negligible subpair of a regular pair is again regular with a similar density, does not hold for degularity. Note that we show in Fact~\ref{fact:degoneside} that a one-sided counterpart of the statement holds, though.

There are a number of complex applications of regularity in which some steps actually utilize only the degularity property rather than the full regularity property. Take for example, `Łuczak's connected matching argument', introduced in~\cite{MR1676887} and used many times since to establish the existence of a long path in a graph $G$ from the fact that its cluster graph $\mathbf{G}$ contains a large matching $\mathbf{M}$ in one connected component.\footnote{`Cluster graph' on the vertex set $\{V_1,\ldots,V_\ell\}$ whose edges correspond to those pairs that are regular, and of sufficient edge density.} The argument has three parts. In the first part, connectivity of $\mathbf{G}$ is used to find a system $\mathcal{S}$ of vertex-disjoint paths in $G$ between each two edges of $\mathbf{M}$. In the second part, each regular pair corresponding to an edge $\mathbf{e}\in\mathbf{M}$ is filled-up with a path $P_\mathbf{e}$ in $G$. In the third part, the paths of $(P_\mathbf{e})_{\mathbf{e}\in\mathbf{M}}$ are merged using $\mathcal{S}$ into a single desired path. While for the second part, one needs the full strength of regular pairs, the system of paths $\mathcal{S}$ in the first part can be planted using a single one-step extension strategy for which the degularity property is sufficient. 
There are examples where degularity is the only property used from an application of the Regularity Lemma. This point is made explicitly in~\cite{MR4604288}, which we quickly and briefly recall. The concept studied in~\cite{MR4604288} is fractional isomorphism. This is a well-studied concept in graphs, and recently a graphon counterpart was introduced in~\cite{MR4482093}. The main result is that if $U$ and $W$ are two fractionally isomorphic graphons, then $U$ can be approximated in the cut distance by a graph $G$ and $W$ can be approximated in the cut distance by a graph $H$ in a way that $G$ and $H$ are fractionally isomorphic. The proof starts with a regularization of a `common equitable partition' of $U$ and $W$.\footnote{The theory of regularizations of graphons is analogous to the theory of regularizations of graphs.} The regularization is one of the two main steps in building $G$ and $H$. Since fractional isomorphism is a concept which is heavily based on degrees and much less on the actual structure, the regularization could be replaced by degularization. We believe that with minor changes, degularity could similarly replace regularity also in~\cite{MR2599196} and~\cite{MR3876899}. 

\subsection{Main result: Lower bound on the number of clusters in degular partitions}
Fact~\ref{fact:regulardegular} tells us that every $\eps$-regular partition is also $2\eps$-degular, that is, for the number
\begin{equation*}
	D(\eps):=\sup\{\min\{\ell\::\:\text{$G$ has an $\eps$-degular partition of complexity $\ell$}\}\;:\;\text{$G$ is a graph}\}
\end{equation*}
we have $D(\eps)\le R(\eps/2)$, yielding in particular $D(\eps)\le\tower(\Theta(\eps^{-5}))$. Given how much weaker the degularity property is compared to regularity --- as we illustrated above and also in Section~\ref{ssec:degpartdeggraph} --- one could hope for a much better bound. Our main result shows that such a hope is in vain.
\begin{theorem}\label{thm:main}
We have $D(\eps)\ge\tower(\Theta(\eps^{-1/3}))$.
\end{theorem}
\subsubsection*{About the proof}
The proof of Theorem~\ref{thm:main} builds heavily on the paper of Moshkovitz and Shapira~\cite{MR3516883} which gives the lower bound $R(\eps)\ge\tower(\Theta(\eps^{-1/2}))$. The main idea of \cite{MR3516883} (which is in turn taken from the the seminal paper~\cite{MR1445389}) for the construction of a graph with no $\eps$-regular partitions of complexity less than $\tower(\Theta(\eps^{-1/2}))$ is the following. We partition the vertex set into a constant number of blobs. Then in $L=\Theta(\eps^{-1/2})$ many rounds, we partition each current blob into a number of subblobs, the number of subblobs being exponential in the number of current blobs. Overall, this results in $\tower(\Theta(\eps^{-1/2}))$ blobs on the $L$-th level. The edges are then ingeniously defined based on this nested blob structure. The key inductive step in the proof of the lower bound is the following property (loosely):
\begin{enumerate}
\item[$(\heartsuit)$] If an $\eps$-regular partition (almost) refines blobs of some level $\ell<L$ then it also (almost) refines blobs on level~$\ell+1$.
\end{enumerate}
This key feature~$(\heartsuit)$ relies, besides the combinatorial definition of the edges, on a basic property of regular pairs, namely that a subpair of a regular pair is regular. As we already illustrated before, this property does not carry over to degular pairs and the same argument therefore does not work to show Theorem~\ref{thm:main}. Moreover in Section~\ref{ssec:ComplexSimple} we discuss that the construction from~\cite{MR3516883} has a $\varepsilon$-degular partition of relatively small complexity. Thus the main work in our proof of Theorem~\ref{thm:main} is to circumvent this shortcoming using a one-sided version of the above subpair property which is still valid in the degular setting, namely Fact~\ref{fact:degoneside}. This in turn leads to introducing additional orientations on the blobs of each level and further refined adjustments of the construction.

\subsubsection*{Structure of the paper}The proof of Theorem~\ref{thm:main} is given in Section~\ref{sec:Proof}. In the rest of this section, we discuss further aspects of our main result. In Sections~\ref{ssec:May1} and~\ref{ssec:May2} we discuss two alternative and looser definitions for degular partitions and argue that a lower bound similar to that in Theorem~\ref{thm:main} holds for them as well. In Section~\ref{ssec:degpartdeggraph} we show that regular graphs have degular partitions of small complexity. We use this to show that each graph is an induced graph of a graph twice as big, and the latter has a degular partition of small complexity. These features are in stark contrast to the setting of the Regularity Lemma. In Section~\ref{ssec:ComplexSimple} we discuss the gap between complexities of degular and regular partitions.

\subsection{Two sparsity conditions for irregular pairs}\label{ssec:May1}
The condition~\ref{P:3} is sometimes called `the degree form of the Regularity Lemma'. There is also an `aggregate form', which instead requires that at most $\eps \ell^2$ pairs $(V_i,V_j)$ are $\eps$-irregular. The degree form obviously implies the aggregate one. However, there a well-known transformation of a $\frac{\eps^2}{8}$-regular partition $U_1\sqcup \ldots \sqcup U_h$ in the aggregate form into the degree form which goes as follows. Remove from the collection each cluster $U_i$ which forms an $\frac{\eps^2}{8}$-irregular pair with more than $\frac{\eps}{4}h$ other clusters. Simple counting gives that there are at most $\frac{\eps}{4}h$ such clusters. We distribute the contents of these problematic clusters evenly over the non-problematic ones. It is easy to check that the resulting partition is $\eps$-regular in the degree form. The same works for degular partitions.

\subsection{Equipartitions versus partitions}\label{ssec:May2}
There is an frequent alternative definition of regularity which does not require equipartitions. That is, the sets $V_1,\ldots,V_\ell$ in~\ref{P:2} may be of arbitrary positive cardinalities. In such a setting~\ref{P:3} would be useless. Indeed we could take $V_1,\ldots,V_{\ell-1}$ singletons and $V_\ell$ the rest. Then all pairs would be regular. Thus,~\ref{P:3} is replaced with the condition
\[
\sum_{i,j:(*)} |V_i||V_j|\ge (1-\eps) n^2\;.
\]
where $(*)$ runs over all $(i,j)\in[\ell]^2$, $i\neq j$ such that $(V_i,V_j)$ is $\eps$-regular. 

One can get from an $\frac{\eps}2$-regular partition $U_1,\ldots,U_\ell$ of clusters of arbitrary cardinalities to the equi-sized setting by breaking up each set $U_i$ into random subsets of cardinality $\frac{\eps n}{2\ell}$ and distributing the leftovers randomly. Note that this way we would only get the aggregate form of the regularity and have to apply the transformation described in Section~\ref{ssec:May1} to transform it to the degree form. The same can be done in the degular setting; we give the details in a related situation in Section~\ref{ssec:degpartdeggraph}.

\subsection{Degular partitions of regular graphs}\label{ssec:degpartdeggraph}
In this section, we show two things. Firstly, we show that each approximately regular graph has a degular partition of arbitrarily small complexity. Secondly, we show that each graph is an induced subgraph of a regular graph of twice the order of the former one, which has by the above a degular partition of arbitrarily small complexity. Both these features fail notably in the setting of regular partitions.

Suppose that $G$ is a regular graph on $n$ vertices. Suppose that $\eps>0$ and $L\in \NN$ are given, with the condition that $L=O_\eps(\log n)$. Basic concentration inequalities give that with positive probability a random partition of $V(G)$ into $L$ sets $V_1,\ldots,V_L$ of individual sizes $\lfloor\frac{n}{L}\rfloor$ and $\lceil\frac{n}{L}\rceil$ has the property that all pairs $(V_i,V_j)$, $\{i,j\}\in\binom{[L]}{2}$, are $\eps$-degular. That is, no non-trivial lower bound holds for the complexity of degular partitions. This can be extended to approximately regular graphs, that is, any graph in which the minimum and the maximum degrees differ by at most~$\frac{\eps n}{2}$. This is in contrast with the situation of regular partitions, see Section~\ref{ssec:ComplexSimple} for details.

This argument can be substantially extended as follows. Suppose that $G$ is an arbitrary graph of order $n$ on vertex set $V$. We construct a graph $H$ of order $2n$ on vertex set $V\sqcup W$. Firstly, we take $H[V]:=G$. Secondly, step-by-step we take each $v\in V$ and by introducing edges between $v$ and $W$ we make sure that $\deg_H(v)=n-1$. The edges are introduced in such a way that in each step we have $\max_{w\in W}\deg(w,V)-\min_{w\in W}\deg(w,V)\le 1$. In the last step, we construct a graph $H[W]$ in such a way that for each $w\in W$ we have $\deg_H(w)=n-1$, or equivalently, $\deg_{H[W]}(w)=n-1-\deg_H(w,V)$. By the above, this is a requirement for a construction of a regular or almost regular graph, and such constructions are well-known to exist.\footnote{As follows from the Erdős--Gallai theorem on graphic sequences.}
To summarize, we have constructed an almost regular graph $H$ on $2n$ vertices of which $G$ is an induced subgraph. By the discussion above, $H$ has a degular partition $V(H)=V_1\sqcup \ldots \sqcup V_\ell$ of arbitrarily small complexity $\ell$ (subject to a trivial condition~\eqref{eq:Eindhoven}). Again, this is in a stark contrast with the situation of regular partitions. Indeed, if the above partition were $\eps$-regular, then the restricted partition $V=(V_1\cap V)\sqcup \ldots \sqcup (V_\ell\cap V)$ would be a $2\eps$-regular partition of $G$ (albeit with~\ref{P:2} dropped; this is known not be a real shortcoming). 

\subsection{Graphs with complex regular partitions but simple degular partitions}\label{ssec:ComplexSimple}
As we show in this paper, even though degularity seems as substantial relaxation of regularity, in general, one cannot use the relaxation to improve much on the construction from Szemer\'edi's regularity lemma. However, there are graphs in which the gap between the most economic $\eps$-regular partition and the $\eps$-degular partition is huge. To see that, recall that Moshkovitz and Shapira~\cite{MR3516883} construct a weighted template graphs of an arbitrarily large order $n$ which is approximately degular\footnote{More precisely, they construct a perfectly degular weighted graph (see Equation~(1) in~\cite{MR3516883}) which is then, by the randomization procedure of Claim~\ref{clm:weightedsel}, transformed into an approximately regular graph. Here, by an `approximately regular graph', we mean that all the degrees differ by at most $\eps n/2$.} but each $\eps$-regular partition has complexity at least $\tower(\Theta(\eps^{-1/2}))$. On the other hand, by the discussion in Section~\ref{ssec:degpartdeggraph}, such graphs have degularizations of complexity $O(\eps^{-1})$.

\section{Proof of Theorem~\ref{thm:main}}\label{sec:Proof}
In this section, we prove Theorem~\ref{thm:main}. More specifically, in Section~\ref{ssec:construction}, for given $\eps$ and $n$ we construct an $n$-vertex \emph{weighted graph}, i.e. a graph where each edge has an assigned $[0,\infty)$-weight, $G(n,s,\delta)$; here we will choose $\delta\sim\eps^{1/3}$ and $s\sim \delta^{-1}$. In Section~\ref{ssec:MainProof} we transform this weighted graph into an unweighted graph and show that each $\eps$-degular partition of this graph has complexity at least $\tower(\Theta(\eps^{-1/3}))$.

There are two auxiliary parts needed. In Section~\ref{ssec:Separators} we construct `separators'. These are certain systems of subsets of a finite auxiliary set, strengthening the notion of `balanced partitions' used in previous work on lower bounds for the Regularity Lemma. Separators are a key building block whose iterative use yields the graph $G(n,s,\delta)$ of our main construction in Section~\ref{ssec:construction}. 

The second auxiliary part is Section~\ref{ssec:densityinheritance}. There, in Fact~\ref{fact:degoneside} we show that the well-known fact that a subpair of a regular pair is again regular has a one-sided counterpart in the degular setting. Fact~\ref{fact:degoneside} is used in the main proof in Section~\ref{ssec:MainProof}.

\subsection{Separators}\label{ssec:Separators}
The iterated key step in our lower-bound construction (as well as in the lower-bound constructions of~\cite{MR1445389} and~\cite{MR3516883}) comprises of the following splitting procedure. Given a subset of the vertices of the host graph, say consisting of $M$ subsets of equal-sized clusters $\mathcal{C}$ of vertices, and an integer $D$, the clusters $\mathcal{C}$ are partititioned in $D$ different ways into two equal-sized collections: the `first' collection and the `second' collection. The key nontrivial property is that of `balancedness': for each two clusters $C_{t},C_{t'}\in \mathcal{C}$, there is a substantial number of partitions (out of the $D$ many above) where $C_t$ belongs into the first collection and $C_{t'}$ belongs to the second collection, or vice versa. The precise definition from~\cite{MR3516883}, in which we identify $\mathcal{C}\doteq[M]$ goes as follows: A sequence of bipartitions $(A_i,B_i)_{i=1}^D$ of $[M]$ is \emph{$c$-balanced} if
\begin{enumerate}[label=(BP\arabic*)]
    \item\label{bal1} for each $i\in[D]$, $A_i$ is a subset of $[M]$ of cardinality $\frac{M}2$, and
    \item\label{bal3} for each $t,t'\in [M]$ distinct, $\left|\left\{i\in[D]:|\{t,t'\}\cap A_i|=1\right\}\right|\ge (\frac{1}{2}-c)D$.
\end{enumerate}
In our construction, we need to add an additional property. This leads us to the notion of separators.
\begin{definition}[separators]\label{def:sep}
Given two even positive integers $D$ and $M$, we call a system $(A_i,B_i)_{i=1}^D$ an \emph{$(M,D)$-separator} if it is an $0.2$-balanced system of bipartitions of $[M]$, and
\begin{enumerate}[label=(S\arabic*)]
    \item\label{sep2} for each $t\in [M]$, $|\{i\in[D]:t\in A_i\}|=\frac{D}2$.
\end{enumerate}
\end{definition}
The existence of a $c$-balanced system of bipartitions (for suitable parameters $D$ and $M$) was shown in~\cite{MR3516883} by a simple application of the probabilistic method. Using a somewhat more complex argument, we can prove the existence of separators.
\begin{lemma}\label{lem:sepexists}
For every even positive integer $D$ and even positive integer $M\le 2^{\lceil D/9999\rceil}$, there exists an $(M,D)$-separator.
\end{lemma}
\begin{proof}
We will only define the sets $A_i$, as the sets $B_i$ are determined by $B_i:=[M]\setminus A_i$.
For $D<9999$ we have $M=2$. We define $A_1=A_2=\ldots=A_{\frac{D}2}=\{1\}$ and $A_{\frac{D}2+1}=A_{\frac{D}2+2}=\ldots=A_D=\{2\}$. Obviously, this satisfies all the required properties.

Suppose that $D\ge 9999$. The construction has two steps. We first construct sets $(A_i)_{i=1}^{D^*}$, where $D^*:=2\lceil0.45D\rceil$, in a way that the following versions of the properties from the definition of balanced systems and separators hold.
\begin{enumerate}[label=(\roman*')]
    \item[(BP1*)] For each $i\in[D^*]$, $A_i$ is a subset of $[M]$ of cardinality $\frac{M}2$.
    \item[(BP2*)] For each $t,t'\in [M]$ distinct, $\left|\left\{i\in[D^*]:|\{t,t'\}\cap A_i|=1\right\}\right|\ge 0.3D$.
    \item[(S1*)] For each $t\in [M]$, $|\{i\in[D^*]:t\in A_i\}|\in [0.41D,0.49D]$.
\end{enumerate}
The proof of this first step is by the probabilistic method. That is, we claim that if taking $A_i$ for $i=1,\ldots,D^*$ as a random subset of $[M]$ of size $\frac{M}2$, the system $(A_i)_{i=1}^{D^*}$ has the desired properties with positive probability. First, we calculate the probability of violating~(S1*) for a single $t\in [M]$. The random variable counting the number of indices $i\in [D^*]$ containing $t$ has mean $\frac{D^*}2=0.45D\pm 0.001D$. Chernoff's bound gives that the probability of this random variable deviating more than $0.039D$ from its mean is at most $2\exp(-2(0.039)^2D)$. Next, we calculate the probability of violating~(BP2*) for a single pair of distinct $t,t'\in [M]$. The random variable counting the number of indices $i\in [D^*]$ containing exactly one of $\{t,t'\}$ has mean $\frac{D^*}2\ge 0.44D$. Chernoff's bound gives that the probability that this random variable is at least $0.14D$ below its expectation is at most $\exp(-2(0.14)^2D)$. We use the union bound over the $M$ many choices of $t$ for which~(S1*) could fail, and also over the $\binom{M}{2}$ many choices of $t$ and $t'$ for which~(BP2*) could fail. Since calculations give for $D\geq 9999$ that 
$$M\cdot 2\exp(-2(0.039)^2D)+\binom{M}{2}\cdot\exp(-2(0.14)^2D)<1\;,$$
we indeed conclude that the system $(A_i)_{i=1}^{D^*}$ has the desired properties~(BP1*), (BP2*) and~(S1*) with positive probability.

We now proceed with the second step of the construction, which is the addition of the sets $(A_i)_{i=D^*+1}^D$. Let $R:=D-D^*$. Now, sequentially in steps $j=1,\ldots,R$, we add a set $A_{D^*+j}$. Before each step $j$, we require the following invariant:
\begin{equation*}
    \mbox{For $t\in [M]$ define $\chi_j(t):=\frac{D}2-|\{i\in[D^*+j-1]:t\in A_i\}|$. We require $\chi_j(t)\le R-j+1$.}
\end{equation*}
That is, $\chi_j(t)$ counts how many times we still need to include the element $t$ in $A_{D^*+j},A_{D^*+j+1},\ldots,A_D$ in order to satisfy~\ref{sep2}. Note that we always have
\begin{equation}\label{eq:MILUJUSUMY}
\sum_{t=1}^M\chi_j(t)=\frac{MD}{2}-(D^*+j-1)\frac{M}{2}=(R-j+1)\frac{M}{2}\;.
\end{equation}
The condition $\chi_j(t)\le R-j+1$ makes sure that we do not have to include $t$ into more sets than the number of leftover steps. Also, note that we may not include an element $t$ into $A_{D^*+j}$ if we have $\chi_j(t)=0$.

First note that this invariant is satisfied for $j=1$. Indeed, by~(S1*), $$\chi_1(t)\le \frac{D}2-0.41D\le 0.1D- 0.001D\le D-D^*=R-1+1\;.$$

Now, in step $j\ge 1$, we define $A_{D^*+j}$ as consisting of the $\frac{M}2$ elements $t\in[M]$ with the highest value of $\chi_j(t)$ (breaking ties arbitrarily). We first verify that none of these selected elements has $\chi_j(t)=0$. As none of the summands  in~\eqref{eq:MILUJUSUMY} is bigger than $R-j+1$ by the invariant, at least $\frac{M}{2}$ many are nonzero. Next, we verify the invariant for $\chi_{j+1}$. Suppose for a contradiction that there exists $t$ such that $\chi_{j+1}(t)\ge R-j+1$. We see that $t$ was not selected for $A_{D^*+j}$ (indeed, otherwise, using the invariant for $\chi_j$, $\chi_{j+1}(t)=\chi_{j}(t)-1\le (R-j+1)-1$). From the general monotonicity $\chi_j\ge \chi_{j+1}$, we have $\chi_{j}(t)\ge R-j+1$. But since $t$ was not selected for $A_{D^*+j}$, we see that there were at least $\frac{M}2$ additional elements with $\chi_{j}$ at least as high. In particular,
$$\sum_{t=1}^M \chi_{j}(t)\ge \left(\frac{M}2+1\right)(R-j+1)\;,$$
a contradiction to~\eqref{eq:MILUJUSUMY}.
\end{proof}

The next lemma is Lemma~2.3 from~\cite{MR3516883}, with changed constants. We omit the proof as it follows \emph{mutatis mutandis}.
\begin{lemma}\label{lem:numbbip}
Let $D,M\in\NN$, and $\zeta\in(0,\frac18)$ be arbitrary. If $(A_i,B_i)_{i=1}^D$ is a sequence of bipartitions of $[M]$ that is $1/5$-balanced, then for every $\lambda=(\lambda_1,\cdots,\lambda_M)$ with $\lambda_t\geq 0$, $\lVert\lambda\rVert_1=1$ and $\lVert\lambda\rVert_\infty\leq 1-8\zeta$, at least $D/15$ of the bipartitions $(A_i,B_i)$ satisfy $\min(\sum_{t\in A_i}\lambda_t,\sum_{t\in B_i}\lambda_t)\geq\zeta$. In particular this holds for an $(M,D)$-separator.
\end{lemma}

\subsection{Construction}\label{ssec:construction}
Given integers $n,s$ and $\delta>0$ we construct a weighted graph $G=G(n,s,\delta)$ on $[n]$. In fact, we construct $s$ weighted graphs $G_1,\ldots,G_s$ and then obtain $G$ by putting these graphs together. Later we will, for a given $\varepsilon>0$, choose $n$ large in terms of $\varepsilon$, $\delta$ and $s$. Since we work with weighted graphs we extend the definition of the degree of a vertex $\deg(v)$ and the edge density between two sets of vertices $d(A,B)$ straightforwardly by summing over the corresponding edge weights. 


Now, we define recursively $m_0=2$ and for $r\in[s]$ we set $m_r=m_{r-1}M_r$, where $M_r=2^{\lceil m_{r-1}/(4\cdot 9999)\rceil}$.
We always assume that $n$ has the right divisibility for what follows and consider a sequence of equipartitions $(\mathcal{X}_r)_{r\in[s]}$ of $[n]$ in which $\mathcal{X}_0=\{\{1,\cdots,n/2\},\{n/2+1,\cdots,n\}\}$ and iteratively $\mathcal{X}_r$ is a refinement of $\mathcal{X}_{r-1}$ by partitioning each partite set of $\mathcal{X}_{r-1}$ into $M_r$ many consecutive intervals. Note that $m_r=|\mathcal{X}_{r}|$.

For every $r\in[s-1]$ we construct a weighted graph $G_{r+1}$. A crucial homogeneity property in our construction will be that for every two distinct cells $C_1,C_2\in \mathcal{X}_{r+1}$,
\begin{equation}\label{eq:homog}
\text{the edge weights of the bipartite graph $G_{r+1}[C_1,C_2]$ are either all equal to $\delta$ or $0$.}
\end{equation}
Write $\mathcal{X}_{r-1}=\{X_1,\cdots,X_{m_{r-1}}\}$ and $\mathcal{X}_{r}=\{X_{i,t}\}_{i\in [m_{r-1}],t\in [M_r]}$ where $X_i$ is partitioned into $\{X_{i,1},\cdots,X_{i,M_r}\}$. For every $1\leq i< j\leq m_{r-1}$ choose a regular bipartite tournament on $\{X_{i,t}\}_{t=1}^{M_r}\cup \{X_{j,t}\}_{t=1}^{M_r}$. Such a bipartite tournament exists; for example we can take the directed edges $$\{(X_{i,t},X_{j,t'})\}_{t\in[M_r],t'\in J^+_t}\cup \{(X_{j,t},X_{i,t'})\}_{t\in[M_r],t'\in J^{++}_t}\;,$$ where $J^+_t\subset [M_r]$ is defined as the $M_r/2$-many elements (recall that $M_r$ is even) that include and follow $t$, viewed cyclically in the set $[M_r]$, $$J^+_t=\{t,t+1,\ldots,t+M_r/2-1\} \mod M_r\;,$$ and similarly, $J^{++}_t$ are the $M_r/2$ many numbers that follow $t$, $$J^{++}_t=\{t+1,t+2,\ldots,t+M_r/2\} \mod M_r\;.$$

\begin{figure}
\includegraphics[scale=0.8]{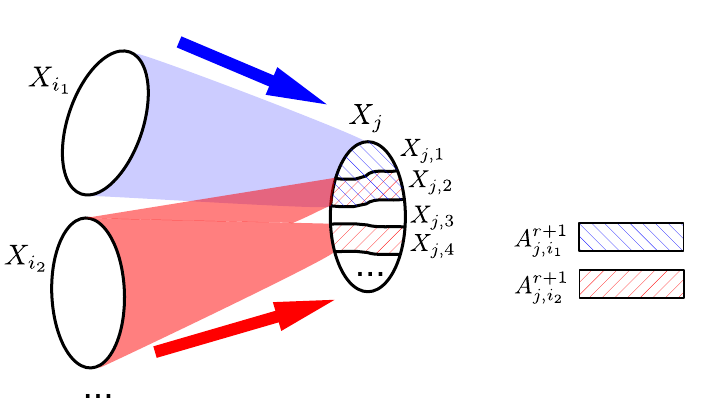}
\caption{A part of the construction of the graph $G_{r+1}$. In this example, $X_{i_1}$ and $X_{i_2}$ are two inneighbours of $X_j$ in $T_r$, and $A'_{i_1}=\{1,2,\cancel{3},\cancel{4},\ldots\}$ and $A'_{i_2}=\{\cancel{1},2,\cancel{3},4,\ldots\}$. The solid blue and red areas indicate positive weights inserted in~\eqref{eq:defGr}.}
\label{fig:Constr}
\end{figure}
The above bipartite tournaments have disjoint edge sets. Their union is denoted by $T_{r}$ and is itself a digraph on $\mathcal{X}_{r}$. Note that the indegree and the outdegree of every vertex in $T_{r}$, i.e. a set $X_{i,t}$, is equal to $D_{r}=(m_{r-1}-1)\cdot M_r/2=(m_r-M_r)/2$. Observe that 
\begin{equation}\label{eq:Dm}
D_r\ge m_r/4    
\end{equation}
and hence $M_{r+1}\le 2^{\lceil D_{r}/ 9999\rceil}$. Figure~\ref{fig:Constr} is aimed to help with the steps that follow. Lemma~\ref{lem:sepexists} tells us there exists an  $(M_{r+1},D_{r})$-separator $(A_j',B_j')_{j=1}^{D_{r}}$. We define a weighted graph $G_{r+1}$ on the vertex set $[n]$ the following way. Fix $j\in m_{r}$. Let $\{X_{i_1},\cdots,X_{i_{D_{r}}}\}$ be the inneighbourhood of $X_j$ in $T_{r}$. Recall that $X_j=\bigcup_{t\in[M_{r+1}]}X_{j,t}$, where $X_{j,t}\in\mathcal{X}_{r+1}$. Then for every $\ell\in\{1,\cdots,D_{r}\}$ set $A^{r+1}_{j,i_{\ell}}=\bigcup_{t\in A'_{\ell}}X_{j,t}$. Finally, set
\begin{equation}\label{eq:defGr}
G_{r+1}(x,y)=\begin{cases}\delta,\textrm{ if } x\in X_i,y\in X_j \textrm{ for some  $(X_i,X_j)\in T_{r}$ and $y\in A^{r+1}_{j,i}$,}\\
0,\textrm{ otherwise.}\end{cases}
\end{equation}
It is clear that~\eqref{eq:homog} holds.

Additionally, for $\mathcal{X}_0=\{X_1,X_2\}$, we define $G_1$ to be the graph with
\begin{equation}\label{eq:defG0}
G_{1}(x,y)=
\begin{cases}0.1,\textrm{ if } x,y\in X_1\;,\\ 
0.9,\textrm{ if } x,y\in X_2\;.
\end{cases}
\end{equation} 

\begin{construction}\label{constr:G}
Given $s\in\mathbb N$, $\delta>0$ and $n\geq m_s$ with $m_s\mid n$, we set $G=G(n,s,\delta)=\sum_{r=1}^s G_r$.
\end{construction}

Given a vertex $x$ and a set of vertices $X$ we write $d(x,X)$ shorthand for $d(\{x\},X)$.

\begin{observation}\label{obs:deg}
For $\delta>0$, $s\in\mathbb N$ and $n\geq m_s$ with $m_s\mid n$, let $G$ be given by Construction~\ref{constr:G}. Let $r\in[s]$ and $X \in\mathcal{X}_{r-1}$. If $x\in [n]\setminus X$, then
$$d_{G_{r+1}}(x,X)+\cdots+ d_{G_s}(x,X)=\tfrac{1}{2}\delta(s-r)\;.$$
\end{observation}

\begin{proof}
Fix $r'\in\{r+1, \ldots,s\}$ and fix the unique $X_i\in\mathcal{X}_{r'-1}$ with $x\in X_i$. For $X_j\in\mathcal{X}_{r'-1}$ we have
\[
d_{G_{r'}}(x,X_j)=
\begin{cases}
0,&\textrm{ if $X_i=X_j$,}\\
\delta/2,&\textrm{ if $(X_i,X_j)\in T_{r'-1}$,}\\
\delta,&\textrm{ if $(X_j,X_i)\in T_{r'-1}$ and $x\in A_{i,j}^{r'}$,}\\
0,&\textrm{ if $(X_j,X_i)\in T_{r'-1}$ and $x\notin A_{i,j}^{r'}$\;.}
\end{cases}
\]
Since $\mathcal{X}_{r'-1}$ is a refinement of $\mathcal{X}_{r-1}$ there exists $\mathcal{P}\subseteq \mathcal{X}_{r'-1}$ such that $X=\bigcup_{X'\in\mathcal{P}}X'$. Furthermore by the definition of $T_{r'-1}$ we have $|\mathcal{P}\cap \NIn_{T_{r'-1}}(X_i)|=|\mathcal{P}\cap \NOut_{T_{r'-1}}(X_i)|=|\mathcal{P}|/2$. Furthermore, by the choice of separators in the construction and Definition~\ref{def:sep}~\ref{sep2} for exactly $\tfrac{|\mathcal{P}|/2}{D_{r'-1}}\cdot\tfrac{D_{r'-1}}{2}=|\mathcal{P}|/4$ many $X_j\in\mathcal{P}\cap \NIn_{T_{r'-1}}(X_i)$ we have $x\in A_{i,j}^{r'}$ and for the same number of $X_j\in\mathcal{P}\cap \NIn_{T_{r'-1}}(X_i)$ we have $x\notin A_{i,j}^{r'}$. Therefore
$$d_{G_{r'}}(x,X)=\tfrac{1}{2}\cdot\delta/2+\tfrac{1}{4}\cdot 0+\tfrac{1}{4}\cdot \delta=\delta/2$$
and the claim follows by summing over $r'=r+1,\ldots,s$.
\end{proof}

\subsection{Density inheritance in degular pairs}\label{ssec:densityinheritance}
A crucial property of an $\varepsilon$-degular pair $(A,B)$ is that the edge-density between every subset $X\subseteq A$ and the whole set of $B$ is almost the same as the one between $A$ and $B$. We summarize this straightforward observation in the following statement.

\begin{fact}\label{fact:degoneside}
Let $(A,B)$ be an $\varepsilon$-degular pair in a possibly weighted graph. For every $X\subseteq A$ we have
\[|d(X,B)-d(A,B)|\leq \left(1+\frac{|A|}{|X|}\right)\varepsilon\;.\]
\end{fact}
\begin{proof}
Let $A^-\subset A$ be a set of size at most $\varepsilon |A|$ as in~\eqref{eq:defdegular}. We have
\begin{align*}
d(X,B)&=\frac{1}{|X|}\left(\sum_{x\in X\setminus A^-}d(x,B)+\sum_{x\in X\cap A^-}d(x,B)\right)\\
&=\frac{1}{|X|}\left(\sum_{x\in X\setminus A^-}(d(A,B)\pm\varepsilon)+\sum_{x\in X\cap A^-}(d(A,B)\pm1)\right)\\
&=d(A,B)\pm\left(\varepsilon+\varepsilon\cdot\frac{|A|}{|X|}\right)\;.
\end{align*}
\end{proof}

\subsection{Main proof}\label{ssec:MainProof}
Below, we work with the natural extension of degularity and degular partitions to the setting of weighted graphs. We show the following lemma which is an adaptation of~Lemma 2.4 from~\cite{MR3516883} to our altered construction. For two vertex sets $A,B$ we write $A\subseteq_{\beta}B$ if $|A\cap B|\geq(1-\beta)|A|$. Furthermore, we say that a partition $\mathcal{Z}$ of $[n]$ is a $\beta$-refinement of a partition $\mathcal{P}$ of $[n]$ if for every $Z\in\mathcal{Z}$ there exists a $P\in\mathcal{P}$ such that $Z\subseteq_{\beta}P$.

\begin{lemma}\label{lem:betaref}
Suppose that $\varepsilon, \beta, \mu, \delta>0$ are such that $32\varepsilon/\mu,4000\varepsilon,1600\beta<\delta<1/2$, and $1\leq r\leq s$ and $s\delta\leq 0.1$. For $n$ large enough let $G(n,s,\delta)$ be given by Construction~\ref{constr:G}. If $\mathcal{Z}$ is an $\eps$-degular partition of $G$ that $\beta$-refines $\mathcal{X}_{r-1}$, then it $(\beta+8\mu)$-refines $\mathcal{X}_{r}$.
\end{lemma}

\begin{proof}
Suppose that there is an $\varepsilon$-degular partition $\mathcal{Z}$ of complexity $\ell$ of the weighted graph $G$ and for the sake of contradiction that there exists $Z_0\in\mathcal{Z}$ with $Z_0\subseteq_{\beta} X_{j^*}$ for some $X_{j^*}\in\mathcal{X}_{r-1}$, but $Z_0\not\subseteq_{\beta+8\mu} X_{j^*,t}$ for every $t\in[M_r]$. From now on we fix this $Z_0$ and $X_{j^*}$.

At the end of the proof, we will show that there are more than $\eps\ell$ many sets $Z\in\mathcal{Z}$ such that $(Z,Z_0)$ is not $\eps$-degular showing that $\mathcal{Z}$ violates the degularity counterpart of~\ref{P:3} and is therefore not an $\varepsilon$-degular partition. We need quite some preparations to obtain this contradiction.

We say that a vertex $v\in [n]$ is \emph{useful} if for the unique $X_i\in\mathcal{X}_{r-1}$ with $X_i\ni v$ and for the unique $Z\in\mathcal{Z}$ with $Z\ni v$ we have $Z\subseteq_{\beta} X_i$. The following claim is crucial for the proof.

\begin{claim}\label{clm:contr}
There are at least $m_{r-1}/120$ many sets $X_i\in\mathcal{X}_{r-1}$ which are \emph{contributing}, i.e. 
\begin{enumerate}[label=(\roman*)]
\item\label{clm:usef} all but at most $(1-120\beta)|X_i|$ many vertices of $X_i$ are useful,
\item $(X_i,X_{j^*})\in T_{r-1}$,
\item\label{Ver3} $|Z_0\cap A_{j^*,i}^r|,|Z_0\cap (X_{j^*}\setminus A^r_{j^*,i})|>\mu|Z_0|$.
\end{enumerate}
\end{claim}
\begin{proof}[Proof of Claim~\ref{clm:contr}]
For each $t\in[M_{r}]$ we set $\lambda_t=|Z_0\cap X_{j^*,t}|/|Z_0\cap X_{j^*}|\geq 0$. Let $\lambda=(\lambda_1,\cdots,\lambda_{M_{r}})$. Note that $\lVert\lambda\rVert_1=1$ and 
$$\lVert\lambda\rVert_\infty<\frac{(1-(\beta+8\mu))|Z_0|}{(1-\beta)|Z_0|}=1-8\cdot\frac{\mu}{1-\beta}\;,$$
as $Z_0\subseteq_{\beta}X_{j^*}$ and $Z_0\not\subseteq_{\beta+8\mu}X_{j^*,t}$ for every $t\in [M_{r}]$. Recall that $|\NIn_{T_{r-1}}(X_{j^*})|=D_{r-1}$ and that the sets $\{A^r_{j^*,i}\}_{X_i\in \NIn_{T_{r-1}}(X_{j^*})}$ are based on an $(M_{r},D_{r-1})$-separator. Hence by Lemma~\ref{lem:numbbip} (with $\zeta=\mu/(1-\beta)$) we have at least $D_{r-1}/15$ many sets $X_i\in \NIn_{T_{r-1}}(X_{j^*})$ such that 
\begin{equation*}
|Z_0\cap A_{j^*,i}^r|,|Z_0\cap (X_{j^*}\setminus A^r_{j^*,i})|\ge \frac{\mu}{1-\beta}\cdot |Z_0\cap X_{j^*}| \ge \mu\cdot |Z_0| \;,
\end{equation*}
as is needed for Item~\ref{Ver3}.

Since $\mathcal{Z}$ is a $\beta$-refinement of $\mathcal{X}_{r-1}$, there are in total at most $\beta n$ many vertices which are not useful. Hence there are at most $m_{r-1}/120$ sets $X_i\in\mathcal{X}_{r-1}$ containing more than $120\beta|X_i|$ non-useful vertices. Thus in total there are at least 
$$D_{r-1}/15-m_{r-1}/120\geByRef{eq:Dm} m_{r-1}/60-m_{r-1}/120=m_{r-1}/120$$
many contributing sets.
\end{proof}

Now consider an arbitrary contributing set $X_{i}\in\mathcal{X}_{r-1}$.
Set $Z_A= Z_0\cap A^r_{j^*,i}$ and $Z_B=Z_0\cap (X_{j^*}\setminus A^r_{j^*,i})$. First note that by~\eqref{eq:homog} we have $d_{G_{r'}}(Z_A,X_{i})=d_{G_{r'}}(Z_B,X_{i})$ for every $r'<r$, as $Z_A$ and $Z_B$ are always part of the same cell in $\mathcal{X}_{r'}$.
Furthermore, we have for every $x_a\in Z_A$ and $x_b\in Z_B$ by~\eqref{eq:defGr} that $d_{G_r}(x_a,X_{i})=\delta$ and $d_{G_r}(x_b,X_i)=0$ and by Observation~\ref{obs:deg} that $\sum_{r'=r+1}^sd_{G_{r'}}(x_a,X_{i})=\tfrac{1}{2}\delta(s-r)=\sum_{r'=r+1}^sd_{G_{i}}(x_b,X_{i})$. Together this leads to
\begin{align}
\begin{split}
\label{eq:degreesum}
&d_G(Z_A,X_i)-d_G(Z_B,X_{i})\\
&~~=d_{G_r}(Z_A,X_{i})-d_{G_r}(Z_B,X_{i})+\sum_{i=r+1}^s\left(d_{G_{i}}(Z_A,X_{i})-d_{G_{i}}(Z_B,X_i)\right)=\delta\;.
\end{split}
\end{align}

Let $\mathcal{Z}_i=\{Z\in\mathcal{Z}\mid Z\subseteq_\beta X_i\}$. By Claim~\ref{clm:contr}\ref{clm:usef} we have
\begin{equation}\label{eq:IamgladImotivatedmyselftoworklateatnight}
|\mathcal{Z}_i|\geq \frac{|X_i|(1-120\beta)}{|Z_0|}=\frac{\frac{n}{m_{r-1}}(1-120\beta)}{\frac{n}{\ell}}=\frac{\ell (1-120\beta)}{m_{r-1}}\;.
\end{equation}

Consider $\widetilde{X}_i=\bigcup_{Z\in\mathcal{Z}_i} Z$ and observe for any $Z'\subseteq Z_0$ that (by using $\beta<1/3200$)
\begin{equation}\label{eq:approxbyuse}
|d_G(Z',X_i)-d_G(Z',\widetilde{X}_i)|\leq \frac{|X_i\triangle\widetilde{X}_i|}{|X_i\cap \widetilde{X}_i|}\leq\frac{120\beta|X_i|+\beta|X_i|}{(1-120\beta)|X_i|}<130\beta\;.
\end{equation}

Define 
$$\mathcal{Z}_i^{\varepsilon}=\{Z\in\mathcal{Z}_i\mid\;(Z,Z_0)\textrm{ is }\varepsilon\textrm{-degular}\}\;.$$
We have
\begin{align*}
\delta&\eqByRef{eq:degreesum}
|d_G(Z_A,X_i)-d_G(Z_B,X_i)|\\
&\overset{\eqref{eq:approxbyuse}}{<}\left|d_G(Z_A,\widetilde{X}_i)-d_G(Z_B,\widetilde{X}_i)\right|+260\beta\\
&\leq \tfrac{1}{|\mathcal{Z}_i|}\sum_{Z\in\mathcal{Z}_i}|d_G(Z_A,Z)-d_G(Z_B,Z)|+260\beta\\
&\leq \tfrac{1}{|\mathcal{Z}_i|}\sum_{Z\in\mathcal{Z}^{\eps}_i}\left(|d_G(Z_A,Z)-d_G(Z_0,Z)|+|d_G(Z_B,Z)-d_G(Z_0,Z)|\right)\\
&\qquad+\tfrac{1}{|\mathcal{Z}_i|}\sum_{Z\in\mathcal{Z}_i\setminus \mathcal{Z}^{\eps}_i}|d_G(Z_A,Z)-d_G(Z_B,Z)|+260\beta\;.
\end{align*}
We can use Fact~\ref{fact:degoneside} on each summand with $Z\in\mathcal{Z}^{\eps}_i$. To this end, we use~Claim~\ref{clm:contr}\ref{Ver3} which tells us that both sets $Z_A$ and $Z_B$ span at least a $\mu$-fraction of $Z_0$. For each summand $Z\in\mathcal{Z}_i\setminus \mathcal{Z}^{\eps}_i$ we use that the densities are between $0$ and $0.9+s\delta\le 1$. Therefore,
\begin{align*}
\delta&{\leq}2\tfrac{|\mathcal{Z}_i^{\varepsilon}|}{|\mathcal{Z}_i|}(\varepsilon+\tfrac{\varepsilon}{\mu})+(1-\tfrac{|\mathcal{Z}_i^{\varepsilon}|}{|\mathcal{Z}_i|})1+260\beta\;.
\end{align*}
Using the assumptions of the lemma, we get $|\mathcal{Z}^\eps_i|\le (1-\frac{1}{4}\delta)|\mathcal{Z}_i|$. Hence, for each contributing $X_i$, we have at least 
\[
\frac{\delta}{4}|\mathcal{Z}_i|\geByRef{eq:IamgladImotivatedmyselftoworklateatnight}\frac{\delta}{5}\ell/m_{r-1}
\]
many pairs $(Z,Z_0)$ with $Z\in\mathcal{Z}_i$ that are not $\eps$-degular. The families $\mathcal{Z}_i$ are mutually disjoint for different contributing $X_i$'s. Therefore by Claim~\ref{clm:contr} we have in total at least $m_{r-1}/120 \cdot \frac{\delta}{5}
\ell/m_{r-1}=\delta\ell/600>\varepsilon\ell$ many non-$\eps$-degular pairs $(Z_0,Z)$ with $Z\in\mathcal{Z}$. Hence, $\mathcal{Z}$ violates the degularity counterpart of~\ref{P:3}, a contradiction.
\end{proof}

We extend the definition of the tower function from base~2 to an arbitrary base $a>1$, $\tower_a:[0,\infty)\to(0,+\infty)$, $\tower_a(x)=1$ for $x\in[0,1)$ and $\tower_a(1+x)=a^{\tower(x)}$ for $x\ge 1$. Note that for every $a>0$ we have $\tower_a(s)\ge \tower(s-O(1))$ as $s\to\infty$.

\begin{theorem}\label{thm:weightedgraph}
There exists $\varepsilon_0$ such that for every $\varepsilon\in (0,\varepsilon_0)$ and every $N\in\mathbb N$ there exists a weighted graph $G$ of order $n\geq N$ such that every non-trivial $\varepsilon$-degular partition $\mathcal{Z}$ of $G$ has complexity at least $\tower_a(s)$ where $a=2^{1/(4\cdot 9999)}$ and $s=\varepsilon^{-1/3}\cdot 10^{-6}$. 
\end{theorem}
\begin{proof}
Set $\delta=\varepsilon^{1/3}$, $\mu=33\varepsilon^{2/3}$ and $s=\lfloor\varepsilon^{-1/3}/10^6\rfloor$. For $\varepsilon$ small enough the following holds. Choose $n\geq \max\{m_s,N\}$ with $m_s\mid n$ and let $G$ be the weighted graph given by Construction~\ref{constr:G} with parameters $\delta$, $s$ and $n$ and let $\mathcal{Z}$ be a non-trivial $\varepsilon$-degular partition of $G$ of complexity $\ell$.
First we show that $\mathcal{Z}$ is an $\varepsilon$-refinement of $\mathcal{X}_0=\{X_1,X_2\}$. Note that by construction and choice of $\delta$ and $s$ we have for every $X\subseteq V(G)$ and $x\in V(G)$ that
\begin{equation}\label{eq:degG0}
\deg_G(x,X)=\sum_{r\in[s]}\deg_{G_r}(x,X)=\deg_{G_1}(x,X)\pm 0.1|X|\;.
\end{equation}
Now suppose that there exists $Z_0\in\mathcal{Z}$ such that $Z_0\not\subseteq_{\varepsilon}X_i$ for $i\in[2]$. Then for every other $Z\in\mathcal{Z}\setminus\{Z_0\}$ let $A^-\subseteq Z_0$ with $|A^-|\leq\varepsilon|Z_0|$ and $B^-\subseteq Z$ with $|B^-|\leq\varepsilon|Z|$ be such that \eqref{eq:defdegular} holds. Since there exist $x_1\in (Z_0\setminus A^{-})\cap X_1$ and $x_2\in (Z_0\setminus A^{-})\cap X_2$ we have by~\eqref{eq:degG0} and by~\eqref{eq:defdegular} for every $Z\in\mathcal{Z}$ such that $(Z_0,Z)$ is $\varepsilon$-degular that
$$0.1|Z\cap X_1|\pm0.1|Z|=\deg(x_1,Z)=\deg(x_2,Z)\pm 2\varepsilon|Z|=0.9|Z\cap X_2|\pm0.1|Z|$$
implying that $|Z\cap X_1|\geq 0.6|Z|$. This leads to
$$|X_1|\geq \sum_{Z\in\mathcal{Z}}|Z\cap X_1|\geq (1-\varepsilon)\ell\cdot \tfrac{0.6n}{\ell}>n/2$$
contradicting that $|X_1|=|X_2|=n/2$.

Note that $s\delta\leq 0.1$, $32\varepsilon/\mu<\varepsilon^{1/3}=\delta$, $4000\varepsilon<\delta$ and for every $r\in[s]$ and $\beta=\varepsilon+8\mu r$ we have $1600\beta=1600\cdot (\varepsilon+8\mu r)<1600\cdot 9\mu s\leq\tfrac{1600\cdot 9\cdot 33}{10^6}\varepsilon^{1/3}<\delta$. Therefore by applying Lemma~\ref{lem:betaref} repeatedly we get that $\mathcal{Z}$ is an $(\varepsilon+r\cdot 8\mu)$-refinement of $\mathcal{X}_r$ for every $r\in[s]$. Hence $\mathcal{Z}$ is a $(\varepsilon+s\cdot 8\mu)\leq1/2$-refinement of $\mathcal{X}_s$ and therefore $|\mathcal{Z}|\geq |\mathcal{X}_s|/2=m_s/2\geq\tower_a(s)$, as was needed.
\end{proof}

Finally, we can show Theorem~\ref{thm:main} by turning the weighted graph from Construction~\ref{constr:G} into an unweighted graph using a simple probabilistic argument.

\begin{proof}[Theorem~\ref{thm:weightedgraph} implies Theorem~\ref{thm:main}]

Given $\eps'>0$ sufficiently small. We set $\eps=4\eps'$ and choose a constant $\zeta$ sufficiently small with respect to $\eps$ and $n$ sufficiently large, i.e.
$$\tfrac{1}{n}\ll\zeta\ll\varepsilon\;.$$

We define $s$ and $a$ as in Theorem~\ref{thm:weightedgraph}. In particular we may assume that $20\zeta^{-2}\log(n)<\varepsilon n/\ell$ for every $\ell\leq \tower_a(s)$. The following claim taken from~\cite{MR3516883} is well-known and follows immediately from Chernoff's inequality.

\begin{claim}[Claim 2.1 from~\cite{MR3516883}]\label{clm:weightedsel}
    Let $\zeta>0$. Suppose $G$ is a weighted complete graph on $n$ vertices with weights in $[0,1]$, and $G'$ is a random graph, where each edge $(x,y)$ is chosen independently to be included in $G'$ with probability $d_G(x, y)$. Then with probability at least $1/2$ we have $|d_{G'}(A, B)-d_G(A,B)|\leq\zeta$ for all sets $A,B$ of size at least $20\zeta^{-2}\log(n)$.
\end{claim}

Let $G$ be given by Theorem~\ref{thm:weightedgraph} such that every non-trivial $\varepsilon$-degular partition of $G$ has complexity at least $\ell>\tower_a(s)$. Let $G'$ be a graph, where each edge $(x,y)$ is chosen independently to be included in $G'$ with probability $d_G(x, y)$, and where we fix an outcome of this random selection such that the assertion of Claim~\ref{clm:weightedsel} holds. Now to show Theorem~\ref{thm:main} it suffices to show that every $\varepsilon'$-degular partition of $G'$ of complexity $\ell\leq\tower_a(s)$ is also an $\varepsilon$-degular partition of $G$ and therefore implying that its complexity is at least $\tower_a(s)$. For this suppose that $(A,B)$ is a $\varepsilon'$-degular pair in $G'$ and such that $|A|=|B|=n/\ell$. We show that $(A,B)$ is also an $4\varepsilon'$-degular pair in $G$. 

Let $A^*:=\{a\in A \::\: |\deg_{G'}(a,B)-\frac{e_{G'}(A,B)}{|A|}|>4\varepsilon'|B|\}$ and $B^*$ be defined similarly. We need to prove that $|A^*|\le 4\varepsilon'|A|$ and $|B^*|\le 4\varepsilon'|B|$. We will show this argument only for $A^*$, the argument for $B^*$ follows \emph{mutatis mutandis}. Suppose for contradiction that $|A^*|> 4\varepsilon'|A|$.

Let $A^{-}\subseteq A$ with $|A^{-}|\leq\varepsilon'|A|$ be such that $\deg_{G}(a,B)=\frac{e_{G}(A,B)}{|A|}\pm\varepsilon'|B|$ for every $a\in A\setminus A^-$.  Without loss of generality we may then assume that there is a set $A'\subseteq A^*$ of size at least $2\varepsilon'|A|$ and such that $\deg_{G'}(a,B)>\frac{e_{G'}(A,B)}{|A|}|+4\varepsilon'|B|$ for every $a\in A'$. Let $\tilde{A}=A'\setminus A^-$. We have $|\tilde{A}|>\eps |A|$ and
\begin{align*}
d_{G}(\tilde{A},B)&=\\
\JUSTIFY{$|\tilde{A}|>\varepsilon' n/\ell>20\zeta^{-2}\log(n)$, Claim~\ref{clm:weightedsel} applies}&=d_{G'}(\tilde{A},B)\pm\zeta\\
\JUSTIFY{$(A,B)$ is $\eps'$-degular in $G'$}&=d_{G'}(A,B)\pm(\varepsilon'+\zeta)\\
\JUSTIFY{Claim~\ref{clm:weightedsel}}&=d_{G}(A,B)\pm(\varepsilon'+2\zeta)<d_{G}(A,B)+4\varepsilon'\\
\JUSTIFY{definition of $\tilde{A}$}&<d_G(\tilde{A},B)\;,
\end{align*}
a contradiction. 
\end{proof}

\section*{Acknowledgments}
We thank Timothy Gowers for discussions. 

\bibliographystyle{plain}
\bibliography{DSG}

\end{document}